\documentclass[12pt,a4paper]{article}
\usepackage[utf8]{inputenc}
\usepackage[english]{babel}
\usepackage{amsmath}
\usepackage{amsfonts}
\usepackage{amssymb}
\usepackage{indentfirst}
\usepackage[usenames,dvipsnames]{color}
\usepackage{hyperref}
\newtheorem{theorem}{Theorem}
\newtheorem{lemma}[theorem]{Lemma}
\newtheorem{proposition}[theorem]{Proposition}
\newtheorem{corollary}[theorem]{Corollary}
\newtheorem{rem}{Remark}
\newenvironment{proof}[1][Proof]{\begin{trivlist}
\item[\hskip \labelsep {\bfseries #1}]}{\end{trivlist}}
\newenvironment{definition}[1][Definition]{\begin{trivlist}
\item[\hskip \labelsep {\bfseries #1}]}{\end{trivlist}}

\author{Thibaut Le Gouic$^{1}$ \& Jean-Michel Loubes$^{2}$ \\ \\ \'Ecole Centrale de Marseille, Institut de Mathématiques de Marseille$^{1}$ \\ Universit\'e de Toulouse, Institut de Math\'ematiques de Toulouse$^{2}$}
\title{Existence and Consistency of Wasserstein Barycenters}
\date{}

\begin{document}
\maketitle

\begin{abstract} 
Based on the Fr\'{e}chet mean, we define a notion of {\it barycenter} corresponding to a usual notion of {\it statistical mean}.
We prove  the existence of Wasserstein barycenters of random probabilities defined on a geodesic space $(E,d)$.
We also prove the consistency of this barycenter in a general setting, that includes taking barycenters of empirical versions of the probability measures or of a growing set of probability measures.
\end{abstract}


\section*{Introduction}

Giving a sense to the  notion of the {\it mean} of a data sample  is one of the major activities of statisticians. When dealing with complex variable data which do not possess an Euclidean structure, the mere issue of defining the mean  becomes a difficult task. This problem arises naturally for a wide range of statistical research fields such as functional data analysis for instance in \cite{Gamboa-Loubes-Maza-07}, \cite{MR2168993}, \cite{2011arXiv1101.0736B}  and references therein, image analysis in \cite{TrouveY05}  or \cite{JASA}, shape analysis in \cite{kendall} or \cite{bb206} with many applications ranging from biology in \cite{Bolstad-03} to pattern recognition \cite{Sakoe-Chiba-78} just to name a few.\vskip .1in
When dealing with data that are probability measures, the issue of finding a central probability measure that will convey the information of the whole data is a difficult task.
This has been tackled in~\cite{agueh2011barycenters} by considering a notion of barycenter  with respect to  the Wasserstein distance. This notion coincides with the notion of Fr\'echet mean. That is, the Fr\'echet mean of the points $(x_i)_{1\leq i \leq n}$ of a geodesic space $(E,d)$ given weights $(\lambda_i)_{1\leq i \leq n}$ is defined as a minimizer of 
\[
x\mapsto \sum_{i=1}^n \lambda_{i}d^2(x,x_i).
\]
This definition provides a natural extension of the barycenter as it coincides on $\mathbb{R}^d$ with the barycenter $\sum_{i=1}^n\lambda_i x_i$ of the points $(x_i)_{1\leq i \leq n}$, with weights $(\lambda_i)_{1\leq i \leq n}$.
This function to be minimized can be rewritten as
\[
x\mapsto \mathbb{E}d^2(X,x)
\]
if the distribution of the random variable $X$  is the discrete measure
\[
\mu = \sum_{i=1}^n \lambda_i \delta_{x_i},
\] where $\delta$ denotes the Dirac measure.
We will call any of these minimizer a barycenter of $\mu$, so that the Fr\'echet mean of $(x_i)_{1\leq i\leq n}$ with weights $(\lambda_i)_{1\leq i\leq n}$ is the barycenter of $\mu$.
There is then a natural extension of a barycenter for a probability measure, that is: a point $x$ is said so be a barycenter of a measure $\mu$ (not necessarily finitely supported) if it minimizes
\[
x\mapsto \mathbb{E}d^2(X,x)
\]
when the distribution of the random variable $X$ is $\mu$.
The first question to arise is whether this barycenter exists.

When $(E,d)$ is assumed to be a locally compact geodesic space, Hopf-Rinow theorem states that balls are compact and thus, the existence of this barycenter is straightforward.
But it is not obvious in more general cases. \vskip .1in

In this paper, we consider barycenters in the Wasserstein space of a locally compact geodesic space.
Since the Wasserstein space of a locally compact space is, in general, not locally compact, its existence is not as straightforward.
However, in this setting, the Wasserstein space is a geodesic space of probability measures.
The first goal of this paper is to prove existence of barycenters in this setting.

Given $p\geq 1$, and denoting $W$ the Wasserstein metric, previous work in this direction consider the barycenter of the probability measures $(\mu_i)_{1\leq i \leq n}$  with weights $(\lambda_i)_{1\leq i \leq n}$, i.e. a minimizer of the following criterion
\[
\nu \mapsto \sum_{i=1}^n \lambda_i W_p^p(\nu,\mu_i),
\]
which is thus also the barycenter of the atomic probability $\mathbb{P}$ on the Wasserstein space, defined by
\[
\mathbb{P}=\sum_{i=1}^n\lambda_i \delta_{\mu_i},
\]
An important result proved in~\cite{agueh2011barycenters}, is the existence and uniqueness of this minimizer when the underlying space $(E,d)$ is the Euclidean space $\mathbb{R}^d$ and $p=2$.
Uniqueness requires the extra assumption that at least one of the $\mu_i$'s \textit{vanishes on small sets}.
This \textit{vanishing} property means that the considered measures give probability $0$ to sets with Hausdorff dimension less than $d-1$.
In particular, any measure absolutely continuous with respect to the Lebesgue measure \textit{vanishes on small sets}.
This work of \cite{agueh2011barycenters} has been extended in~\cite{2014arXiv1412.7726K} to compact Riemannian manifolds, with the condition to \textit{vanish on small sets} being replaced by absolute continuity with respect to the volume measure.  
Since the Wasserstein space of a compact space is also compact, the existence of the barycenter in this setting is straightforward, but their work provides, among other results, an interesting extension of the work of \cite{agueh2011barycenters}, by showing a dual problem called the \textit{multimarginal} problem, for any $\mathbb{P}$ of the form $\sum_{i=1}^n \lambda_i \delta_{\mu_i}.$
The same dual problem has been used in a previous work to show existence of barycenter whenever there exists a Borel (not necessarily unique) barycenter application on $(E^n,d^n)$ that associate the barycenter of $\sum_{i=1}^n\lambda_i \delta_{\mu_i}$ to every $n$-uplets $(x_1,...,x_n)$.
This assumption is actually always verified on locally compact geodesic spaces.
This is the result of Lemma~\ref{lem:borelbarycenter}.
It is a first step toward the proof of existence of barycenter for any $\mathbb{P}$.  \vskip .1in

This paper studies, in the setting of locally compact geodesic spaces, the existence of the barycenter and state consistency properties. In a previous work \cite{Bernoulli} or \cite{these}, the authors studied  some asymptotic results giving conditions under which a sequence of barycenters of discrete measures converging to a limit measure can be understood as a barycenter of the limit probability measure.
This result enables to define the barycenter of empirical measures and study its asymptotic behavior.
In the following, we propose an improved version of this limit theorem that enables to prove existence of barycenters of probabilities in a our more general framework.   \vskip .1in

This paper falls into the following parts.
Section~\ref{s:def} presents general definitions and states a general theorem that ensures the existence of a barycenter of probability measures in the Wasserstein space.
In Section~\ref{s:consistency}, a consistency result is proven. Section~\ref{s:appli} is devoted to some statistical applications.
The technical lemmas are presented in Section~\ref{s:lemma} while the detailed proofs are postponed to Section~\ref{s:append}.

\section{Barycenter of a probability in Wasserstein space} \label{s:def}

Given two points $x,y$ in a metric space $(E,d)$, their \emph{mid-point} is the point $z\in E$ such that 
\[
d(x,z)=d(z,y)=\frac{1}{2}d(x,y).
\]
\begin{definition}[Definition (Geodesic space)]
A space $(E,d)$ is called a \emph{geodesic space} if
\begin{itemize}
\item $(E,d)$ is a \emph{complete metric} space and
\item every two points $x,y\in E$ have a mid-point $z\in E$.
\end{itemize}
\end{definition}
Note that in this case, the mid-point of $x$ and $y$ is the 2-barycenter of $x$ and $y$ with weights $(\frac{1}{2},\frac{1}{2})$.

\begin{rem}
Such spaces are sometimes called \emph{complete intrinsic metric length spaces}  (see for instance \cite{bbi}).

Given a continuous path $\gamma:[0,T]\rightarrow E$, its length is defined as 
\[
\Lambda(\gamma)=\sup \left\{\sum_{i=0}^n d\left(\gamma(t_{i+1}),\gamma(t_i)\right);0=t_0\leq t_1 \leq ... \leq t_n=T\right\}.
\]
Thus, a continuous path is said to be a \emph{geodesic}, if for any interval $[a,b]\subset [0,T]$, the length of $\gamma$ restricted to $[a,b]$ is $d\left(\gamma(a),\gamma(b)\right)$:
\[
\Lambda(\gamma_{|[a,b]})=d\left(\gamma(a),\gamma(b)\right).
\]

It is known (see theorem 2.4.16 p.42 and lemma 2.4.8 p.41 in \cite{bbi}) that a (separable) complete metric space is geodesic if and only if for every pair $(x,y)$ there exists a geodesic joining $x$ and $y$.
\end{rem}

\begin{definition}[Definition (Barycenter)]
Set $p\geq1$ and let $(E,d)$ be a geodesic space and $\mu$ a probability measure on $(E,d)$ such that 
\begin{equation}\label{eq:Pp}
\int d^p(x,x_0)d\mu(x)<\infty
\end{equation}
for some (and thus any) $x_0\in E$.
A point $x_0\in E$ is called a \emph{$p$-barycenter} of $\mu$ if 
\begin{equation}\label{eq:defbary}
\int d^p(x,x_0)d\mu(x) = \inf\left\{\int d^p(x,y)d\mu(x);y\in E\right\}.
\end{equation}
The set of all probability measures satisfying (\ref{eq:Pp}) is denoted $\mathcal{W}_p(E)$.
\end{definition}

Barycenters do not always exists.
On can find a geodesic space $(E,d)$ and a probability measure $\mu\in \mathcal{W}_p(E)$ for which there exists no barycenter.
However, the Hopf-Rinow-Cohn-Vossen theorem (see theorem 2.5.28 p. 52 in \cite{bbi}) states that, on locally compact geodesic spaces, every closed ball is compact.
Consequently, the infimum in (\ref{eq:defbary}) can be taken on a compact ball, and thus existence of a barycenter is ensured.
We thus have the following proposition.
\begin{proposition}
Set $p\geq 1$ and let $(E,d)$ be a \emph{locally compact} geodesic space and $\mu\in \mathcal{W}_p(E)$.
Then, there exists a barycenter of $\mu$.
\end{proposition}

Metric spaces of nonpositive curvature (NPC spaces) provide another setting for which barycenters exist.
We recall the definition of such spaces following \cite{sturm2003}.

\begin{definition}[Definition (NPC Spaces)]
A complete metric space $(E,d)$ is called a \emph{global NPC space} if for each pair of points $x_0,x_1\in E$, there exists $y\in E$ such that for all $z\in E$,
\[
d^2(z,y)\leq \frac{1}{2}d^2(z,x_0) + \frac{1}{2} d^2(z,x_1) -\frac{1}{4} d^2(x_0,x_1).
\]
\end{definition}

Such spaces are geodesic spaces and every probability measure on such spaces that satisfies  $\int d^2(x,x_0)d\mu(x)<\infty$ for some $x_0\in E$ has a \emph{unique} $2$-barycenter (see proposition 4.3 in \cite{sturm2003}).

The goal of this paper is to study barycenters in Wasserstein spaces.
We first recall the definition of the Wasserstein space of a metric space $(E,d)$.

\begin{definition}[Definition (Wasserstein space)]
Set $p\geq 1$ and let $(E,d)$ be a metric space.
Given two measures $\mu$, $\nu$ in $\mathcal{W}_p(E)$, we denote by $\Gamma(\mu, \nu)$ the set of all
probability measures $\pi$ over the product set $E \times E$  with first, resp. second, marginal $\mu$, resp. $\nu$.
The transportation cost with cost function $d^p$ between two measures
$\mu$, $\nu$  in $\mathcal{W}_p(E)$, is defined as
\begin{equation*}
 \mathcal{T}_p(\mu, \nu) = \inf_{ \pi \in \Gamma(\mu, \nu)} \int d^p(x, y) d \pi.
\end{equation*}
The  transportation cost allows to endow the set $\mathcal{W}_p(E)$  with a metric $W_p$ defined by
\begin{equation*}
 W_p(\mu, \nu) = \mathcal{T}_p(\mu, \nu)^{1/p}.
\end{equation*}
This metric is known as the $p$-Wasserstein distance and the metric space $(\mathcal{W}_p(E),W_p)$ is called the \emph{Wassertein space} of $(E,d)$.
\end{definition}

It is well known (see theorem 6.9 of \cite{villani2008optimal} for instance, or proposition 7.1.5 in \cite{ambrosio2008}) that $W_p$ metrizes the topology of weak convergence and convergence of moments of order $p$ (i.e. $\int d^p(x,x_0)d\mu_n \rightarrow \int d^p(x,x_0)d\mu(x)$).
If $(E,d)$ is a separable complete metric space, so is $(\mathcal{W}_p(E),W_p)$ (see theorem 6.19 in \cite{villani2008optimal}).

Also, if $(E,d)$ is a locally compact geodesic space, then $(\mathcal{W}_p(E),W_p)$ is a geodesic space.
This result can be found in \cite{lott2009} (see lemma 2.4 and proposition 2.6) for the case $(E,d)$ compact, but the arguments are valid when $(E,d)$ is locally compact as well.
However, $(\mathcal{W}_p(E),W_p)$ is not locally compact unless $(E,d)$ is compact (see remark 7.1.9 in \cite{ambrosio2008}).
Thus, existence results on locally compact spaces can not be applied to prove existence of barycenter on $(\mathcal{W}(E),W_p)$.

Likewise, Wasserstein spaces are not NPC spaces in general (see theorem 7.3.2 in \cite{ambrosio2008}).
Indeed, two probability measures $\mu_0, \mu_1$ can have more than one mid-point in $(\mathcal{W}(E),W_p)$: each mid-point is a barycenter of $\frac{1}{2}\left(\delta_{\mu_0}+\delta_{\mu_1}\right)\in \mathcal{W}(\mathcal{W}(E))$.
However NPC spaces assign a \emph{unique} barycenter to every probability measure.
Therefore Wasserstein spaces can not be NPC spaces and results for such spaces can not be applied to prove existence of barycenter in Wasserstein spaces.

We want to prove existence of barycenters in Wasserstein spaces.
To that purpose, we consider a random probability measure $\tilde{\mu}$ in $\mathcal{W}_p(E)$, following a distribution $\mathbb{P}$. 
This probability $\mathbb{P}$ is chosen in the space $\mathcal{W}_p(\mathcal{W}_p(E))$ endowed with the metric $W_p$.
Note that we use the same notation for the Wasserstein distance over $\mathcal{W}_p(E)$ and $\mathcal{W}_p(\mathcal{W}_p(E))$.
Thus, if $\tilde{\mu} \in \mathcal{W}_p(E)$ is a random measure with distribution 
 $\mathbb{P}$, then for all $\nu \in \mathcal{W}_p(E)$, we can write  
\begin{equation} \label{equiv}  W^p_p(\delta_\nu,\mathbb{P}) = \mathbb{E} (W^p_p(\nu,\tilde\mu)) = \int W^p_p(\nu,\mu) d\mathbb{P}(\mu). \end{equation}

For a probability $\mathbb{P} \in \mathcal{W}_p(\mathcal{W}_p(E))$, consider a minimizer over $\nu \in \mathcal{W}_p(E)$ of 
\[
 \nu \mapsto \mathbb{E} \left[ W_p^p(\nu,\tilde\mu) \right]=W^p_p(\delta_\nu,\mathbb{P}), 
\]
where $\tilde\mu$ is a random probability of $\mathcal{W}_p(E)$ with distribution $\mathbb{P}$. If exists, this probability measure is a barycenter of $\mathbb{P}$.

We can now state existence result.

\begin{theorem}[Existence of a Wasserstein Barycenter] \label{th:bar}
Set  $p\geq 1$ and let $(E,d)$ be a separable locally compact geodesic space.
Hence, for $\mathbb{P} \in \mathcal{W}_p(\mathcal{W}_p(E))$, there exists a barycenter $\bar{\mu}_\mathbb{P}$ defined as 
\begin{equation} \label{bar}
\bar{\mu}_\mathbb{P} \in {\rm arg}\min_{ \nu \in \mathcal{W}_p(E) }  \mathbb{E}  \left[ W_p^p(\nu,\tilde\mu) \right],
\end{equation} for $\tilde \mu$ a random measure with distribution $ \mathbb{P}$.
\end{theorem}

Using the expression \eqref{equiv},  we can see that Theorem~\ref{th:bar} can be reformulated as stating the existence of the metric projection of $\mathbb{P}$ onto the subset of $\mathcal{W}_p(\mathcal{W}_p(E))$ of Dirac measures.  

\begin{proof}
The proof of Theorem~\ref{th:bar} relies on the existence of barycenters of finitely supported measures in $\mathcal{W}_p(E)$ for which the core ideas were developped in \cite{agueh2011barycenters}.
Those ideas are used for the first step of this proof.
The proof is split in three steps.
\begin{itemize}
\item  First,
consider a set of probability measures $(\mu_j), {j=1,\dots,J}$ of $\mathcal{W}_p(E)$ and assume that 
 $\mathbb{P}$ is a discrete measure defined, for positive weights $ \lambda_1,\dots, \lambda_J$  such that $\sum_{j=1}^J \lambda_j=1$, as $$\mathbb{P}=\sum_{j=1}^J \lambda_j \delta_{\mu_j}.$$ In this case 
\[
W^p_p(\delta_\nu,\mathbb{P})=\mathbb{E} W_p^p(\nu,\tilde\mu)=\sum_{j=1}^J \lambda_j W_p^p(\nu,\mu_j).
\]
Within this framework, Theorem~\ref{th:bar} reduces to an already solved problem, in the case $p=2$ in~\cite{agueh2011barycenters} or \cite{Bernoulli} and for general $p$ in~\cite{these}.
It is recalled in this paper as Theorem~\ref{thm:casfini}.
\item  To prove the theorem in the general case, we show that if there is a sequence of probability measures  $({\mathbb{P}_j})_{j\geq 1}$ converging to a limit probability measure $\mathbb{P}$  and if for each $\mathbb{P}_j$ there exists a barycenter $\bar{\mu}_{\mathbb{P}_j}$, then there exists a barycenter $\bar\mu_{\mathbb{P}}$ of the limit probability $\mathbb{P}$.
Moreover $\bar\mu_{\mathbb{P}}$ is the limit of a subsequence of the barycenters  $(\bar{\mu}_{\mathbb{P}_j} )_{j\geq 1}$.
This result is stated as Theorem~\ref{prop} in the following section.
\item Finally Proposition~\ref{prop2} concludes the proof showing that one can approximate any probability measure in $\mathcal{W}_p(\mathcal{W}_p(E))$  by finitely supported probability measures.
\end{itemize}
\end{proof}

\section{Consistency of the barycenter of a sequence of measures} \label{s:consistency}

The following theorem deals with  a continuity issue of the barycenters.
Consider a sequence $(\mathbb{P}_j)_{j\geq 1}\subset\mathcal{W}_p(E)$ converging to some $\mathbb{P}$ in $\mathcal{W}_p(\mathcal{W}_p(E))$.
If these measures all admit a barycenter, it is natural to ask whether the sequence of barycenters also converges to a barycenter of $\mathbb{P}$.
Theorem~\ref{prop} provides a positive answer.

\begin{theorem}\label{prop}
Set  $p\geq 1$ and let $(E,d)$ be a separable locally compact geodesic space.
Let $(\mathbb{P}_j)_{j \geq1} \subset \mathcal{W}_p(\mathcal{W}_p(E))$ be a sequence of probability measures on  $\mathcal{W}_p(E)$ and set $\mu_j$ a barycenter of $\mathbb{P}_j$, for all $j\in\mathbb{N}$. 
Suppose that for some $\mathbb{P}  \in \mathcal{W}_p(\mathcal{W}_p(E)) $, we have that $W_p(\mathbb{P},\mathbb{P}_j)\stackrel{j \rightarrow + \infty}{\longrightarrow} 0$. Then, the sequence $(\mu_j)_{j\geq 1}$ is precompact in $\mathcal{W}_p(E)$ and any limit is a barycenter of $\mathbb{P}$.
\end{theorem}

\begin{proof}[Sketch of proof]
The proof of Theorem~\ref{prop} can be split into three steps.
\begin{itemize}
\item The first step shows that the sequence of barycenters  $(\mu_j)_{j\geq 1}$ is tight. It is a consequence of the fact that balls on $(E,d)$ are compact together with  Markov's inequality  applied to these balls.

\item The second step uses Skorokhod representation theorem and lower semicontinuity of $\nu\mapsto W_p(\mu,\nu)$ for any $\mu$, to show that any weak limit of the sequence $(\mu_j)_{j\geq 1}$ is a barycenter of $\mathbb{P}$.

\item The final step proves that the convergence of the $(\mu_j)_{j\geq 1}$ actually holds in $\mathcal{W}_p(E)$.
\end{itemize}
\end{proof}

Applying this result to a constant sequence gives the following corollary.

\begin{corollary}
The set of all barycenters of a given measure $\mathbb{P}\in\mathcal{W}_p(\mathcal{W}_p(E))$ is compact.
\end{corollary}

An interesting and immediate corollary follows from the assumption that $\mathbb{P}$ has a unique barycenter.

\begin{corollary}
Suppose $\mathbb{P}\in\mathcal{W}_p(\mathcal{W}_p(E))$ has a unique barycenter. Then for any sequence $(\mathbb{P}_j)_{j\geq 1}\subset\mathcal{W}_p(\mathcal{W}_p(E))$ converging to $\mathbb{P}$, any sequence $(\mu_j)_{j\geq 1}$ of their barycenters converges to the barycenter of $\mathbb{P}$.
\end{corollary}

On $E=\mathbb{R}^d$ and $p=2$, there exists a simple condition under which the barycenter is unique.

\begin{proposition}
Let $\mathbb{P}\in\mathcal{W}_2(\mathcal{W}_2(\mathbb{R}^d))$ such that there exists a set $A\subset\mathcal{P}_2(\mathbb{R}^d)$ of measures such that for all $\mu\in A$,
\begin{equation}\label{eq:L}
B\in\mathcal{B}(\mathbb{R}^d), \dim(B)\leq d-1 \implies \mu(B)=0,
\end{equation}
and $\mathbb{P}(A)>0$, then, $\mathbb{P}$ admits a unique barycenter.

Therefore, for any sequence $(\mathbb{P}_j)_{j\geq 1}$ converging to $\mathbb{P}$ in $\mathcal{W}_2\mathcal{W}_2(\mathbb{R}^d))$, the barycenters of $\mathbb{P}_j$ converge to the barycenter of $\mathbb{P}$.
\end{proposition}

\begin{proof}
It is a consequence of the fact that if $\nu$ satisfies (\ref{eq:L}), then $\mu\mapsto W_2(\mu,\nu)$ is strictly convex and thus, so is $\mu\mapsto \mathbb{E}W_2^2(\mu,\tilde\mu)$.
\end{proof}

\section{Statistical applications} \label{s:appli}

\subsection*{Two statistical frameworks}

When confronted to the statistical analysis of a collection of probability measures in $\mathcal{W}_p(E)$, $\mu_1,\dots,\mu_J$, it is natural to define a notion of variability as 
\[
{\rm V}_J(\mu_1,\dots,\mu_J)=\inf_{ \nu \in\mathcal{W}_p(E) } \frac{1}{J} \sum_{j=1}^J W_p^p(\nu,\mu_j).
\]
This quantity plays the  role of a variance which measures the spread, with respect to the Wasserstein distance, of the measures around a point which is the Wasserstein barycenter. In this work, we extend this definition to match the notion of variance by defining 
\[
{\rm V}(\mu)=\inf_{ \nu \in\mathcal{W}_p(E) } \mathbf{E} \left( W_p^p(\nu,\tilde\mu) \right),
\]
where $\tilde\mu$ is a random probability measure in $\mathcal{W}(E)$.
We provide some condition that ensures that this quantity is well defined and is achieved for a measure $\bar{\mu}_{\mathbb{P}}$, which plays the role of the \emph{mean} of the random measure $\tilde\mu$.
Moreover, statistical inference in this setting has been tackled in  two different frameworks whether the number of probabilities goes to infinity or whether the probabilities are not observed directly but through empirical samples.
Theorem~\ref{prop} handles both of these settings. \\
\indent The first point of view concerns the case where the distribution $\mathbb{P} \in \mathcal{W}_p(\mathcal{W}_p(E))$ is approximated by a growing discrete distribution $\mathbb{P}_J$ supported on $J$ elements, with $J$ growing to infinity.
Consider a collection of measures $\mu_j\in\mathcal{W}_p(E)$ for $j \geq 1$, and weights $\lambda^J_j  \geq 0$, and define the sequence of measures  $\mathbb{P}_J, J \geq 1$ as follows
\[
\mathbb{P}_J=\sum_{i=1}^J\lambda_i^J\delta_{\mu_j}.
\]
Assume that $\mathbb{P}_J$ converges to some measure $\mathbb{P}$ with respect to Wasserstein distance.
Hence Theorem~\ref{prop} states that the barycenter (or any barycenter if not unique) of $\mathbb{P}_J$ converges to the barycenter of $\mathbb{P}$ (provided $\mathbb{P}$ has a unique barycenter). \\
\indent  The second  asymptotic point of view deals with the case where the measures $\mu_j$ are unknown but approximated by a sequence of measures $\mu_j^n$ converging with respect to the Wasserstein distance to measures $\mu_j$ when $n$ grows to infinity. Compared to the first framework, the number of measures here is fixed but only an estimation of the measures is known. This covers the interesting  case where we observe i.i.d sample $X_{i,j}$ with $i=1,\dots,n$ with distribution $\mu_j \in \mathcal{W}_p(E)$. Here $\mu_j^n =\frac{1}{n} \sum_{i=1}^n \delta_{X_{i,j}}$ is the empirical measure.
Given positive weights $(\lambda_i)_{1\leq i\leq J}$ (or a sequence of weights converging to them) the issue is whether the barycenter of the observed measure $\sum_{j=1}^J\lambda_j \delta_{\mu_j^n}$ converges to the barycenter of the limit $\sum_{j=1}^J\lambda_j \delta_{\mu_j}$ in the case where this barycenter is unique. This problem has been answered positively in \cite{Bernoulli}, up to extracting a subsequence, since the barycenter is not unique. Within this framework, set
\[
\mathbb{P}_n=\sum_{j=1}^J\lambda_j\delta_{\mu_j^n}
\]
with positive weights $\lambda_j$ and measures $(\mu_j^n)_{1\leq j\leq J,n\geq 1}\subset\mathcal{W}_p(E)^J$ converging to some limit measures $(\mu_j)_{1\leq j \leq J}\in\mathcal{W}_p(E)^J$. Then Theorem~\ref{prop} states that the barycenter (or any if not unique) converges to the barycenter of $\sum_{j=1}^J\lambda_j\delta_{\mu_j^n}$ (if unique).\vskip .1in

\subsection*{Implications of the results}

The existence  and consistency of Wasserstein barycenters has several implications in statistics.\vskip .1in
First, the variance of a collection of measures is helpful to understand the separation between collections of probability measures. Goodness of fit testing procedures have been developed to assess similarity between two samples as in \cite{MR1704844} or \cite{Loubes2015}. The test statistics relies on the computation of the variance of the sample.  Its calculation uses its expression that involves  the computation of the Wasserstein distance of the distribution with respect to the mean of the probability, which is obtained by proving the existence of the mean distribution.
\vskip .1in

Then, one of the major application is given by   deformation models or registration issues of distributions. In these problems, one assume that an unknown  template distribution $\mu$ is warped from different observations by a random deformation process. The goal is here to estimate the template using the observations. 
More precisely, the probability measured $\mu_j$ are warped from the template by a random center deformation operator $T$ with realizations $T_j$, such that $$\mu_j= {T_j}_{\#} \mu = \mu \circ T_j^{-1}.$$
Then the barycenter of the $\mu_j$'s is a proper estimate for the unobserved template. In a previous work~\cite{Bernoulli}, a similar result has been proved under a more restrictive assumption on the $\mu_j$'s: this result was proven in the case when $E=\mathbb{R}^d$ and the $(\mu_j)_{j\geq 1}$ are \textit{admissible deformations} in the sense that they can be written as the pushforward of a common probability measure $\mu$ by the gradient of a convex function.
This setting has also been considered in \cite{MR3306430}.  In~\cite{2012arXiv1212.2562B} this problem is also tackled in the particular case where the $(\mu_j)_{j\geq 1}$ have  compact support, are absolutely continuous with respect to the Lebesgue measure and  are indexed on a compact set $\Theta$ of $\mathbb{R}^d$. They state more precisely that given a probability measure on $\Theta$, one can induce a probability measure $\mathbb{P}$ on $\mathcal{W}_p(\mathbb{R}^d)$, and if the $(\mu_j)_{j\geq 1}$ are chosen randomly under $\mathbb{P}^{\otimes \infty}$, the (unique) barycenter of $\frac{1}{J}\sum_{j=1}\delta_{\mu_j}$ converges to the barycenter of $\mathbb{P}$, $\mathbb{P}$-almost surely.
In our case, we handle the general case where the family of deformations is a random function which induces a random distribution $\mu_j$ with distribution $\mathbb{P}$ given by the law of the deformation, which enables to consider general random deformation models.
Natural applications in biology arise when dealing with gene expressions that suffer from a huge variability due to the different ways of processing the data.
The first task preliminary to any analysis is a normalization procedure to extract a mean feature which corresponds to the mean distribution or the  Wasserstein barycenter as proved in~\cite{Bolstad-03} or~\cite{GALLON-2011-593476}. 
In all these cases, our result provides the existence of the {\it target} mean distribution while  the consistency results allow the barycenters to be approximated by taking the barycenter of noisy data sample.
 \vskip .1in
In a more general way, finding a way to combine complex information from several sources is a problem that is receiving a growing interest, in particular when the data  can be modeled as random distributions or samples of distributions. It is the case in Big data when we want  to  exploit  massive  data  sets  that  could  have  been collected by different units or that exceed the size to make feasible their analysis on a single machine.
Hence inference on such data sets can not be conducted using all the data, and the barycenter of the distributions is a natural candidate to resume the information conveyed by all the data.
In this framework the barycenter distribution plays the role of a consensus distribution that could represent a consensus-based global estimation or confidence region. This point of view is developed in \cite{2015arXiv151105350A} where the mean of the distribution is chosen as a representative distribution. Similar cases are considered in information fusion, where the goal is similar since it amounts to finding a mean measure that aggregates the information provided by different input measures. Hence the Wasserstein barycenter is a natural way to aggregate this information as pointed out in~\cite{bishop2014information}. In multi-target tracking, the main issue is the estimation of both the number and locations of multiple moving targets such as airplanes based on sensor measurements. In~ \cite{baum2015wasserstein} the Wasserstein barycenter provides an alternative to the MOSPA (Mean Optimal Sub-Pattern Alignment) distance. \\
\indent Finally, when considering Bayesian inference, one is faced with the problem of approximating a posterior measure.
Such approximation can be done by sampling posterior from the data. 
For large data sets, computation of such samples becomes intractable.
Thus, one can split the data into small subsets and combine the results of these local computations.
Taking the mean of the Bayesian posterior measures provides a natural way to combine these local computations as pointed out in \cite{2015arXiv150805880S}.
\vskip .1in

\section{Proofs} \label{s:append}

\begin{lemma}[Borel barycenter application]\label{lem:borelbarycenter}
Set $p\geq 1$ and let $(E,d)$ be a separable locally compact geodesic space.
Then, given any $J\in\mathbb{N}^*$ and weights $(\lambda_j)_{1\leq j \leq J}$, there exists a Borel application $T:E^J\rightarrow E$ that associate $(x_j)_{1\leq j \leq J}$ to a minimum of $x\mapsto \sum_{j=1}^J\lambda_j d^p(x,x_j)$.
Such applications will be called \emph{Borel barycenter application}.
\end{lemma}

\begin{proof}[Proof of Lemma~\ref{lem:borelbarycenter}]
Since $(E,d)$ is locally compact, applying theorem A.5 in \cite{zimmer2013} with $X=E^J$, $Y=E$ and 
\[
A=\left\{(x_1,...,x_J,x)\in X\times Y; \sum_{j=1}^J\lambda_j d^p(x,x_j)\leq \sum_{j=1}^J\lambda_j d^p(z,x_j) \forall z\in E\right\},
\]
shows the existence of a Borel section $f$ from $\pi_X(A)$ to $X\times Y$ of the projection $\pi_X:X\times Y\rightarrow X$.
Then $T=\pi_Y\circ f$ is a Borel barycenter application - where $\pi_Y:X\times Y \rightarrow Y$ denotes the projection.
\end{proof}

\begin{theorem}[Barycenter and multi-marginal problem]\label{thm:casfini}
Let $(E,d)$ be a complete separable geodesic space, $p\geq 1$ and $J\in\mathbb{N}^*$. Given $(\mu_i)_{1 \leq i \leq J} \in \mathcal{P}_p(E)^J$ and weights $(\lambda_i)_{1 \leq i \leq J}$, there exists a measure $\gamma \in \Gamma(\mu_1, ... , \mu_J)$ minimizing
\[
\hat{\gamma} \mapsto \int \inf_{x \in E} \sum_{1 \leq i \leq J}  \lambda_i d(x_i,x)^p d\hat{\gamma}(x_1,...,x_J).
\]
Moreover, denote $T:E^J \rightarrow E$ a Borel barycenter application (as in Lemma~\ref{lem:borelbarycenter}) then the measure $\nu=T_\# \gamma$ is a barycenter of $(\mu_i)_{1 \leq i \leq J}$ and if this application is unique, any barycenter $\nu$ is of the form $\nu=T_\# \gamma$.
\end{theorem}

\begin{proof}[Proof of Theorem~\ref{thm:casfini}]
This proof is adapted from proposition 4.2 of \cite{agueh2011barycenters}.

Existence of the solution of the multi-marginal problem is a direct consequence of lemma \ref{compacit_marginales}.

Denote by $\gamma$ a solution of the multi-marginal problem and set $\nu=T_\# \gamma$.
Then, by definition of the Wasserstein distance,
\[
W_p^p (\mu_i,\nu) \leq \int d^p(x_i,T(x)) d\gamma(x).
\]
and consequently, 
\begin{equation}\label{nu_opt}
\sum_{1 \leq i \leq J} \lambda_i W_p^p (\mu_i,\nu) \leq \int \sum_{1 \leq i \leq J} \lambda_i d^p(x_i,T(x)) d\gamma(x).
\end{equation}

Also, for $\hat{\nu} \in \mathcal{P}_p(E)$, denote $\pi_i \in \Gamma(\mu_i,\hat{\nu})$ the optimal transport plan between $\hat{\nu}$ and $\mu_i$.
Using disintegration theorem, for any $1 \leq i \leq J$, there exists a (conditional) measure $\mu_i^y$ defined for $\hat{\nu}-$almost any $y$, which satisfies $\pi_i(x,y) = \mu_i^y(x)\otimes \hat{\nu}(y)$.
Set then,
\[
\theta(x,y)=\mu_1^y(x_1)\otimes ... \otimes \mu_J^y(x_J) \otimes\hat{\nu}(y),
\]
and denote $\hat{\gamma}$ the law of the $J$ first marginals of $\theta$. Then, by construction of $\theta$,
\begin{align}
\sum_{1 \leq i \leq J} \lambda_i W_p^p (\mu_i,\hat{\nu}) = & \sum_{1 \leq i \leq J} \lambda_i \int d^p(x_i,y)d\theta(x,y)\nonumber\\
 = &  \int \sum_{1 \leq i \leq J} \lambda_i d^p(x_i,y)d\theta(x,y)\nonumber\\
 \geq &  \int \inf_{z \in E} \sum_{1 \leq i \leq J} \lambda_i d^p(x_i,z)d\theta(x,y)\label{ssi_T}\\
= &  \int \sum_{1 \leq i \leq J} \lambda_i d^p(x_i,T(x))d\theta(x,y)\nonumber\\
= &  \int \sum_{1 \leq i \leq J} \lambda_i d^p(x_i,T(x))d\hat{\gamma}(x)\nonumber\\
\geq &  \int \sum_{1 \leq i \leq J} \lambda_i d^p(x_i,T(x))d\gamma(x)\nonumber\\
\geq & \sum_{1 \leq i \leq J} \lambda_i W_p^p (\mu_i,\nu),\nonumber
\end{align}
where the last inequality is an application of (\ref{nu_opt}).

Since $\hat{\nu}$ is arbitrary, we have just shown that $T_\#\gamma$ is a barycenter.

Also, taking $\hat{\nu}$ a barycenter, (\ref{ssi_T}) becomes an equality, so that for $\theta-$almost any $(x,y)\in E^J \times E$,
\[
\sum_{1 \leq i \leq J} \lambda_i d^p(x_i,y) =\inf_{z \in E} \sum_{1 \leq i \leq J} \lambda_i d^p(x_i,z)=\sum_{1 \leq i \leq J} \lambda_i d^p(x_i,T(x)),
\]
and thus, if the barycenter application $T$ is unique, $T(x)=y, \theta-$almost surely and so $T_\# \hat{\gamma}=\hat{\nu}$. 
Also, optimality of $\hat{\nu}$, and (\ref{nu_opt}) show that $\hat{\gamma}$ is a solution of the multi-marginal problem.

\end{proof}

\begin{proof}[Proof of Theorem~\ref{prop}]
Denotes $\mu_j$ a barycenter of $\mathbb{P}_j$.
The proof is in three steps.
\begin{enumerate}
\item Proving the tightness of the sequence of the barycenters  $(\mu_j)_{j\geq 1}$.
\item Proving that any limit $\mu$ of $(\mu_j)_{j\geq 1}$ (in the sense of the weak convergence of measures) is a barycenter.
\item Proving that there exists $\nu \in \mathcal{W}_p(E)$ such that $W_p(\nu,\mu_j)\rightarrow W_p(\nu,\mu)$. The conclusion of the proof will be derived from Lemma~\ref{lem:cvW}.
\end{enumerate}
Let $\tilde\mu$ and $\tilde\mu_j$ random measures with distribution respectively $\mathbb{P}$ and $\mathbb{P}_j$.

\begin{enumerate}
\item First prove that the moments of order $p$ of the random measures considered as random variables $\mu_j$ can be bounded from above by a constant  $M<\infty$.

Let $\tilde\mu_j$ be a random measure drawn according to a  distribution  $\mathbb{P}_j$. Then, for any $x\in E$ 
\begin{align*}
W_p(\mu_j,\delta_x) &=W_p(\delta_{\mu_j},\delta_{\delta_x})\\
&\leq W_p(\delta_{\mu_j},\mathbb{P}_j) + W_p( \mathbb{P}_j,\delta_{\delta_x})\\
&= \left(EW_p^p(\mu_j,\tilde\mu_j)\right)^{1/p}+ \left(EW_p^p(\tilde\mu_j,\delta_x)\right)^{1/p}\\
&\leq 2 \left(EW_p^p(\tilde\mu_j,\delta_x)\right)^{1/p} \text{ since $\mu_j$ is a minimizer of $\nu \mapsto EW_p^p(\nu,\tilde\mu_j)$}\\
&= 2 W_p(\mathbb{P}_j,\delta_{\delta_x})\\
&\leq 2 \left( W_p(\mathbb{P}_j,\mathbb{P}) + W_p(\mathbb{P},\delta_{\delta_x})\right) \leq M <\infty \text{ since } W_p(\mathbb{P},\mathbb{P}_j)\rightarrow 0.
\end{align*}
Denote $B(x,r)$ the ball of $E$ centered in  $x$ with radius $r$. Then Markov's inequality entails that 
\[
\mu_j(B(x,r)^c) \leq \frac{E_{\mu_j}d^p(X,x)}{r^p} = \frac{W_p^p(\mu_j,\delta_x)}{r^p} \leq \frac{M^p}{r^p}.
\]
The compactness of the balls of  $E$ entails that the sequence $(\mu_j)_{j\geq 1}$ is tight. So it can be extracted a sequence which converges towards a distribution that will be denoted $\mu$. For ease of notations, the subsequence will be denoted as the initial sequence.

\item Let $\nu \in \mathcal{W}_p(E)$ and $\tilde\mu$ a random measure with distribution  $\mathbb{P}$. We get
\begin{align}
EW_p^p(\nu,\tilde\mu) &= W_p^p(\delta_\nu,\mathbb{P})\nonumber\\ 
&=  \lim_{j\rightarrow \infty}  W_p^p(\delta_\nu,\mathbb{P}_j)\text{ since $W_p(\mathbb{P}_j,\mathbb{P}) \rightarrow 0$}\nonumber\\
&= \lim_{j\rightarrow \infty} EW_p^p(\nu,\tilde\mu_j)\nonumber\\
& \geq \lim_{j \rightarrow \infty} EW_p^p(\mu_j,\tilde\mu_j) \text{ since $\mu_j$ is a barycenter}\label{eq:egalite}\\
& \geq E\liminf_{j\rightarrow \infty} W_p^p(\mu_j,\tilde\mu_j) \text{ using Fatou's Lemma for any coupling of the  $\tilde\mu_j$'s}\nonumber\\
& \geq EW_p^p(\mu,\tilde\mu) \text{ since $W_p$ is lower semi-continuous}.\nonumber
\end{align}
For the last inequality, we used that since $\mathbb{P}_j \rightarrow \mathbb{P}$, Skorokhod's representation theorem enables to build  $\tilde\mu_j\rightarrow \tilde\mu$ a.s..
This proves that  $\mu$ is a barycenter of $\mathbb{P}$.

\item For $\nu=\mu$, the inequality~(\ref{eq:egalite}) is in fact an equality which implies that
\[
W_p(\delta_{\mu_j},\mathbb{P}_j) \rightarrow W_p(\delta_{\mu},\mathbb{P}).
\]
Hence
\[
W_p(\delta_{\mu_j},\mathbb{P}) - W_p(\delta_{\mu},\mathbb{P}) \leq W_p(\delta_{\mu_j},\mathbb{P}_j) + W_p(\mathbb{P}_j,\mathbb{P}) - W_p(\delta_{\mu},\mathbb{P}) \rightarrow 0.
\]
This implies that
\begin{align*}
EW_p^p(\mu,\tilde\mu)&=W_p(\delta_{\mu},\mathbb{P})\\
&=\lim_{j\rightarrow \infty} W_p(\delta_{\mu_j},\mathbb{P})\\
&= \lim_{j\rightarrow \infty} EW_p^p(\mu_j,\tilde\mu)\\
& \geq E \liminf_{j\rightarrow \infty} W_p^p(\mu_j,\tilde\mu) \text{ using Fatou's Lemma}\\
& \geq EW_p^p(\mu,\tilde\mu) \text{ using again semi-lower continuity of  $W_p$ for weak convergence.}
\end{align*}
So $\mathbb{P}$-a.s, (since $ \liminf W_p(\mu_j,\tilde\mu)\geq W_p(\mu,\tilde\mu)$)
\[
\liminf W_p(\mu_j,\tilde\mu) = W_p(\mu,\tilde\mu).
\]
So all along a subsequence and for a $\nu \in \mathcal{W}(E)$,  $W_p(\mu_j,\nu) \rightarrow W_p(\mu,\nu)$. So using Lemma~\ref{lem:cvW}, we get that
\[
W_p(\mu_j,\mu)\rightarrow 0,
\]
which concludes the proof.
\end{enumerate}

\end{proof}

\section{Technical Lemmas} \label{s:lemma}

The following five results are well known.
They are recalled here for the purpose of clarity of the proofs.

\begin{lemma}[Consistency in  $L^1$]\label{lem:cvL1}
Let $(X_n)_{n\geq 1}$ be a sequence of real valued random variables such that 
\begin{align*}
X_n\rightarrow X \text{ a.s.}\\
\mathbb{E}|X_n|\rightarrow \mathbb{E}|X|.
\end{align*}
Then, $X_n\rightarrow X$ in $L^1$.
\end{lemma}
\begin{lemma}[Uniform integrability]\label{lem:UI}
A family of real valued random variables  $\mathcal{H}$ is uniformly integrable  (in the sense that  $\sup_{X\in \mathcal{H}} \int_{\{|X|>a\}}|X|d\mathbb{P}\rightarrow 0$ as $a\rightarrow +\infty$) if and only if the two following conditions hold 

i) $\sup_{X \in \mathcal{H}} E|X| <\infty$ (bounded in $L^1$)

ii) $\forall \varepsilon>0, \exists \alpha>0$ such that $\forall A\in\mathcal{A}, \left(\mathbb{P}(A)<\alpha \implies \sup_{X\in\mathcal{H}} \int_A|X|d\mathbb{P}<\varepsilon\right)$ (equicontinuity).

\end{lemma}

\begin{lemma}[Consistency in  $L^1$ and uniform integrability ]\label{lem:UIL1}
Let $X_n\rightarrow X$ in probability, then 
the sequence $(X_n)_{n\geq 1}$ is uniformly integrable if and only if  $X_n\rightarrow X$ in $L^1$.
\end{lemma}

\begin{lemma}[Tightness of fixed marginals set of measures]\label{compacit_marginales}
Let $(E,d)$ be a Polish space (i.e. a complete separable metric space). Let $C_1, ..., C_J$ be compacts sets of $\mathcal{W}_p(E)$. Then, the set $\Gamma(C_1, ..., C_J)$ defined as the set of probability measures on $E^J$ with marginals respectively in $C_1,...,C_J$, is compact.

\end{lemma}

\begin{proposition}[Approximation by finitely supported measures] \label{prop2}
For all $\mathbb{P}$ there is a sequence of finitely supported probabilities $\mathbb{P}_j$ such that $$ W_p(\mathbb{P}_j,\mathbb{P}) \longrightarrow 0.$$
\end{proposition}

Here is a lemma used for the proof of Theorem~\ref{prop}.

\begin{lemma}\label{lem:cvW}
Let $(\mu_n)_{n\geq 1}$ be a sequence of measures on a Polish space $(E,d)$ which converges weakly towards  $\mu$. If there exists a measure  $\nu$ such that 
\[
W_p(\mu_n,\nu)\rightarrow W_p(\mu,\nu),
\]
then
\begin{equation} \label{conv}
W_p(\mu_n,\mu)\rightarrow 0.
\end{equation}
\end{lemma}
\begin{proof}
Note first that if $\nu=\delta_x$ for a given  $x\in E$, then \eqref{conv} is true, due to the fact that Wasserstein convergence is equivalent to the weak convergence plus convergence of the order $p$ moments (see \cite{villani2008optimal}).

First using the Gluing Lemma (see for instance in \cite{villani2008optimal} or \cite{berti2014gluing}), build three sequences  $(X_n)_{n\geq 1},(Y_n)_{n\geq 1},(Z_n)_{n\geq 1}$ with distribution respectively  $\mu_n$, $\nu$ and $\mu$ such that 
\[
(X_n,Y_n)\sim \pi^1_n,(Y_n,Z_n)\sim \pi^2_n,
\]
where $\pi^1_n$ and  $\pi^2_n$ are the optimal transport maps between respectively $\mu_n$ and $\nu$ and between  $\nu$ and $\mu$.
Let $\Pi_n$ be the distribution of $(X_n,Y_n,Z_n)$. 
Since the three marginals weakly converge, the sequence $(\Pi_n)_{n\geq 1}$ is tight.
Thus, we can extract a subsequence such that  
\[
\Pi_n \rightarrow \Pi \text{ weakly,}
\]
where $\Pi$ has marginal probabilities $\mu,\nu$ and $\mu$. 

Then  Skorokhod's representation Theorem enables to construct a space  $(\Omega,\mathcal{A},\mathbb{P})$ on which there exist  $X,Y,Z$ with joint distribution $\Pi$ and copies of $(X_n,Y_n,Z_n)$ with law $\Pi_n$ such that
\[
d(X_n,X)+d(Y_n,Y)+d(Z_n,Z)\rightarrow 0 \text{ $\mathbb{P}$-a.s..}
\]
If we show that $(d^p(X_n,X))_{n\geq 1}$ is uniformly integrable then using Lemma~\eqref{lem:UIL1}, we get  $$\mathbb{E}d^p(X_n,X)\rightarrow 0,$$ which implies the result since  $W_p(\mu_n,\mu)\leq \mathbb{E}d^p(X_n,X)$.

Uniform integrability remains to be proven. Note that Lemma~\eqref{lem:UI} entails that it is equivalent to prove the two following assumptions 

\begin{enumerate}
\item[i)] $\sup_{n\geq 1} \mathbb{E}d^p(X_n,X) <\infty$ (bounded in $L^1$)

\item[ii)] $\forall \varepsilon>0, \exists \alpha>0$ such that $\forall A\in\mathcal{A}, \left(\mathbb{P}(A)<\alpha \implies \int_Ad^p(X_n,X)d\mathbb{P}<\varepsilon\right)$ (equicontinuity).
\end{enumerate}
Assertion i) is a consequence of - since $\mathbb{E}d^p(X_n,X) \leq \mathbb{E}d^p(X_n,Z_n)$,
\begin{align*}
\mathbb{E}d^p(X_n,X) &\leq C_p\left[\mathbb{E}d^p(X_n,Y_n)+\mathbb{E}d^p(Y_n,x) + \mathbb{E}d^p(x,Z_n)\right] \\
&= C_p\left(W_p^p(\mu_n,\nu) + W_p^p(\nu,\delta_x) + W_p^p(\delta_x,\mu)\right)\\
& \leq M <\infty \text{ since we assumed that  }W_p^p(\mu_n,\nu) \rightarrow W_p^p(\mu,\nu).
\end{align*}
To prove Assertion ii), set $A\in \mathcal{A}$. We have that
\begin{equation}\label{eq:equicont}
\mathbb{E}d^p(X_n,X)\mathbf{1}_{ A} \leq C_p \left[\mathbb{E}d^p(X_n,Y_n)\mathbf{1}_{ A}+\mathbb{E}d^p(Y_n,x)\mathbf{1}_{ A} + \mathbb{E}d^p(x,Z_n)\mathbf{1}_{ A}\right].
\end{equation}
Note that $d^p(X_n,Y_n)$, $d^p(Y_n,x)$ and $d^p(x,Z_n)$ converge towards respectively  $d^p(X,Y)$, $d^p(Y,x)$ and $d^p(x,Z)$ a.s. Their $L^1$ norm converge also, for the first term by assumption and since  $Y_n$ and $Z_n$ are identically distributed, for all $n\geq 1$. Hence using  Lemma~\ref{lem:cvL1} they converge in $L^1$ and thus are equicontinuous sequences. Hence this implies that for all  $\varepsilon>0$, there exists $\alpha>0$ such that the three terms $$\mathbb{E}d^p(X_n,Y_n)\mathbf{1}_{ A}+\mathbb{E}d^p(Y_n,x)\mathbf{1}_{ A} + \mathbb{E}d^p(x,Z_n)\mathbf{1}_{ A}<3\varepsilon$$ for any $A$ such that $\mathbb{P}(A)<\alpha$. 

Hence inequality \eqref{eq:equicont} implies that $(d^p(X_n,X))_{n\geq 1}$ is equicontinuous. Since it is also bounded in  $L^1$, this sequence is uniformly integrable, which proves the result.
\end{proof}

\bibliographystyle{plain}
\bibliography{bibli}
\end{document}